\documentclass[12pt, reqno ]{amsart}
\usepackage{amsmath,amssymb,amsthm,verbatim,color}
\usepackage[english]{babel}
\usepackage{amsfonts}
\usepackage{amssymb}
\usepackage{pgfplots}
\pgfplotsset{compat=1.18}

\usepackage[letterpaper,top=2in,bottom=1.5in,left=1.5in,right=1.5in,marginparwidth=1.75in]{geometry}

\usepackage{amsmath}
\usepackage{graphicx}
\usepackage{mathtools}
\usepackage{amssymb}
\usepackage{amsthm}
\usepackage{float}
\usepackage[style=alphabetic]{biblatex}

\usepackage[colorlinks=true, allcolors=blue]{hyperref}

\newtheorem{theorem}{Theorem}[section]
\newtheorem{prop}[theorem]{Proposition}
\newtheorem{definition}[theorem]{Definition}

\newtheorem{lemma}[theorem]{Lemma}
\numberwithin{equation}{section}

\newtheorem{conjecture}[theorem]{Conjecture}
\newtheorem{ques}[theorem]{Question}
\newcommand{\mycomment}[1]{}

\title{Minimal Surfaces near Hardt-Simon foliations}
\author{Vishnu Nandakumaran}
\address{Department of Mathematics, University of Notre Dame, IN 46637} \email{vnandaku@nd.edu}
\begin{document}
\thanks{ The authors declare no competing interests.\\
No new data were created or analysed in this study. Data sharing is not applicable to this article.}
\maketitle
\begin{abstract}
    Caffarelli-Hardt-Simon \cite{CHS} used the minimal surface equation on the Simons cone $C(S^3\times S^3)$ to generate newer examples of minimal hypersurfaces with isolated singularities. Hardt-Simon \cite{SimonHardt} proved that every area-minimizing quadratic cone $\mathcal{C}$ having only an isolated singularity can be approximated by a unique foliation of $\mathbb R^{n+1}$ by smooth, area-minimizing hypersurfaces asymptotic to $\mathcal C$. This paper uses methods similar to \cite{CHS} to solve the minimal surface equation for the Hardt-Simon surfaces in the sphere for some boundary values. We use gluing methods to construct minimal surfaces over Hardt-Simon surfaces and near quadratic cones. 
\end{abstract}

\setcounter{tocdepth}{1}
\tableofcontents

\section{Introduction}

Let $N^n\subset \mathbb R^{n+1}$ be a minimal hypersurface, it is a critical point of the area functional. One of the key ways to study the local and global structure of these minimal surfaces is through the study of its tangent cones at 0 and infinity, and thereby, the need to study area-minimizing cones. Due to the works of Bernstein, Fleming \cite{Fleming1962OnTO}, De Giorgi \cite{DeGiorgi1965}, Almgrem \cite{almgren1966}, and Simons \cite{SimonJ}, we know that only area minimizing cones for $n\leq 6$ are flat planes. The first non-trivial area-minimizing cone was the Simons cone due to Simons \cite{SimonJ} and Bombieri-De Giorgi-Giusti \cite{Bombieri1969}. Thereafter, generalized Simons cones or quadratic cones $C^{p,q}$ (see Definition \ref{prel:defn:Quadratic cones}), which are cones on the product of spheres, are one of the most well-understood examples of area-minimizing minimal cones. A natural question is if we can perturb these minimal cones to obtain examples of hypersurfaces with singularities. This led to the study of the graphical minimal surface equation on the Simons cone. A first question one could ask is,
\begin{ques}\label{Q1}
    For minimal quadratic cones $\mathcal{C}^{p,q}\subset \mathbb R^{n+1}$ and $\Sigma=\mathcal C^{p,q}\cap \partial B_1$ be its link where $B_1$ is the closed unit ball. Given $g\in C^{2, \alpha}(\Sigma)$ with $\|g\|_{C^2, \alpha}$ small and let $\Sigma'=graph_\Sigma(g)$, then
    \begin{enumerate}
        \item \label{Q1.1} is there a minimal hypersurface $N$ graphical over $\mathcal C^{p,q}\cap B_1$ such that $\Sigma'\subset N $. 
        \item if not, can we study these surfaces and comment on their graphicality over nearby "nice" surfaces.  
    \end{enumerate} 
\end{ques}
In an attempt to find a plethora of examples of minimal surfaces with isolated singularities \cite{CHS} studies a boundary value type problem for the minimal surface equation on $\mathcal C^{3, 3}$. They note this can be answered for a finite co-dimension set. More specifically, for $g\in V^{\perp}\subset C^{2, \alpha}(\Sigma)$ there is a unique $u$ on $\mathcal C^{p,q}\cap B_1$, graph of $u$ is minimal and $u|_{\Sigma}=g$, with required decay estimates (see Theorem \ref{CHSthm}). Here $V$ is a finite-dimensional set corresponding to the first three Fourier modes on $\Sigma$. \\
Hardt-Simon \cite{SimonHardt} proved that for quadratic cones, there exists a smooth minimal hypersurface $S_+$ contained completely on one side of $\mathcal C^{p,q}$ and is unique up to scaling. For $\lambda >0$, $S_\lambda=\lambda S_+$ forms a foliation of one of the components of $\mathbb R^{n+1}\backslash \mathcal C^{p,q}$ (see \cite{SimonHardt}). Recently, there have been other results similar to Hardt-Simon addressing other minimizing cones and generic regularity (see \cite{wang2022}, \cite{chodosh2023genericregularityminimizinghypersurfaces}, \cite{https://doi.org/10.15781/z70n-ed29}).
\\ Outside a compact set, the Hardt-Simon surface $S_+$ is a graph over $\mathcal C^{p,q}\backslash B_{R_0}$, for $R_0>0$ and over $\Sigma$ it is the graph of a constant function. $S_\lambda\cap B_1$ is a clear obstruction to the first part of the question and corresponds to one of the dimensions in the set $V$ apart from translations and rotations. The uniqueness of Hardt-Simon's result motivates us to believe that the Hardt-Simon surface is the only key obstruction in studying perturbations near quadratic cones. So the main result of this paper is to study graphical solutions near the Hardt-Simon surface similar to \cite{CHS} and finally use this to study perturbation of $\mathcal C^{p,q}$.\\
In order to study minimal surfaces near $\mathcal C^{p,q}$, the paper's primary focus is to study minimal surfaces near leaves of the foliations $S_\lambda$. For a function $u$ on $S_\lambda$, let $\mathcal M_\lambda(u)$ denote the mean curvature of the graph of $u$ over $S_\lambda$. We prove the following theorem, 
\begin{theorem}[c.f. Theorem \ref{1}]\label{intro:1}
For $\lambda \ll 1$ let $S_\lambda$ denote a leaf of the foliation on $\mathcal C^{p,q}$, let $\Sigma_\lambda\coloneqq S_\lambda \cap \partial B_1$ and $g\in C^{2, \alpha}(\Sigma_\lambda)$, then for $\|g\|_{C^{2, \alpha}}$ small and independent of $\lambda$, then there exists a function $u\in C^{2, \alpha}_\delta(S_\lambda \cap B_1) $ solving the boundary value problem
    \begin{alignat*}{2}
       \mathcal{M}_\lambda u&=0&\quad &\text{on }S_{\lambda}\cap B_1,\\
                \Pi u&=\Pi g &      &\text{on }\Sigma_\lambda.
\end{alignat*} 
Here $\Pi$ is the projection away from a finite-dimensional set $V$ due to the first three Fourier modes (see Definition \ref{defn:Pi} and Theorem \ref{1}). 
\end{theorem}
The above theorem provides examples for minimal surfaces near the leaves of the foliation. For the cone $\mathcal C^{p,q}$ the finite-dimensional set corresponds to translations, rotations, and the Hardt-Simon foliations (see Theorem \ref{eigenvector}  and Section \ref{sec:HSFoliations}).  This construction, in turn, allows us to construct minimal surfaces near quadratic cones for arbitrarily small boundary perturbations. We use the above theorem to prove the following result, it is related to the main result of Edelen-Spolaor \cite{Nick}. Our result is more like a Dirichlet boundary value problem and follows from their result. 
\begin{theorem}[c.f. Theorem \ref{MainTheo}]
\label{intro:2}
Let $g\in C^{2, \alpha}(\Sigma)$ and $\|g\|_{C^{2,\alpha}}$ is small, then there exists a minimal hypersurface $N^n\subset \mathbb R^{n+1}$, and $(a, Q, \lambda)\in\mathbb R^{n+1}\times SO(n+1)\times \mathbb R$ such that $N$ is graphical over $(a+Q(S_\lambda))\cap B_1$and $graph_{\Sigma_\lambda}(g)\subset N$. 
\end{theorem}
The naive idea is that, modulo the three operations translations, rotations, and dilations ($S_\lambda, \lambda\in \mathbb R$), we can solve the boundary value problem. In this paper, we do not understand the choice $(a, Q, \lambda)$ well enough to comment on anything stronger. We expect that this choice $(a, Q, \lambda)$ is unique and continuous with respect to $g$ (refer Conjecture \ref{conjecture}). The eventual goal is to achieve better regularity for the choice to understand the formation of singularities better and the Plateau's problem near these minimizing cones.
\subsection{Organization of the paper}
Section 2 introduces some preliminaries and sets up the required notation. We shall define weighted spaces over the Simons cone and study the Hardt-Simon foliation. Section 3 defines an approximate surface $\tilde S_\lambda$ (Definition \ref{defn:approxsurface}). The approximate surface is a glued version of the $S_\lambda$ to the corresponding quadratic cone $\mathcal C$ and will be very close to the $S_\lambda$ for $\lambda<<1$. In this section, we shall set up and study the basic notions on  $\tilde S_\lambda$. Finally, in section 4, we shall study the Dirichlet type boundary problem on $\tilde S_\lambda$ and prove Theorem \ref{intro:1} and subsequently Theorem \ref{intro:2}. We expect a stronger uniqueness result (Conjecture \ref{conjecture}) and work out a uniqueness-type result in the more straightforward case involving only translations and rotations. 

\subsection*{Acknowledgment}
I am deeply grateful to my graduate advisor, G\'abor Sz\'ekelyhidi, for his constant support, patience, and guidance. I would like to thank the Department of Mathematics, University of Notre Dame, for the financial support through my graduate studies. I am also thankful to co-advisor Nicholas Edelen and Sourav Ghosh for discussing various topics and helping me better understand the subject.\\
I would also like to thank the referee for the detailed comments, which helped in improving the correctness and clarity of the paper. 

\section{Preliminaries and Notation}\label{sec:prelim}

\begin{definition}
We say $N^n\subset \mathbb R^{n+1}$ is a hypersurface with singular set $S$ if $N$ is a multiplicity one current in $\mathbb R^{n+1}$ and the set of interior regular points $reg(N)$ is an $n-$dimensional orientable hypersurface in $\mathbb R^{n+1}$ and $S=Sing(N)\coloneqq N\backslash Reg(N). $ 
\end{definition}
We shall only work with hypersurfaces containing isolated singularities, for which the above definition should be sufficient. However, one would generally need to work with currents and varifolds. Refer to Simon \cite{simon1984lectures} and Federer \cite{federer2014geometric} for more on currents and varifolds. \\
The graph of a function $u$ on $N^n$ denoted by $graph_N(u)$ is defined as
$$graph_N(u)\coloneqq \{x+u(x)\nu_N(x): x\in N\},$$
where $\nu_N$ is the oriented unit normal on $N$. We shall also use spherical graphs but shall define them ahead. \\
For a function $u$ on $N$, let $\mathcal{M}_N(u)$ denote the mean curvature of $graph_N(u)$. The function $u$ being a critical point of the area functional is equivalent to the mean curvature $\mathcal M_N (u)=0$,  and the corresponding surfaces are called minimal surfaces. In summary, Minimal surfaces correspond to critical points of the area functional.  \\
The mean curvature $\mathcal{M}_N(u)$ is a $2^{nd}$ order quasi-linear differential equation, of the form
$$\mathcal{M}_N(u)=a_N(x,u,\nabla u)^{ij}\nabla^2_{ij}u+b_N(x,u,\nabla u).$$
We shall often call $\mathcal M_N$ the \textit{minimal surface equation}. 
Let $\mathcal{L}_N(u)$ be the linearization of $\mathcal{M}_N$ at $0$, i.e., for $u:N\rightarrow \mathbb R$, 
$$\left. \mathcal{L}_N(u)=\frac{d}{dt}\right |_{t=0} \mathcal{M}_N(tu) .$$  In general, $\mathcal{L}_N$ is given by 
\begin{equation}\label{L_N}
\mathcal{L}_N=\Delta_N+|A_N|^2,
\end{equation}
here $A_N$ is the second fundamental form of $N^n\subset \mathbb R^{n+1}$, and $\Delta_N$ is the Laplacian operator in $N$.
Denote the higher-order terms as, 
$$Q_N(u)\coloneqq \mathcal{M}_N(u)-\mathcal{L}_N(u).$$
\begin{definition}
    Given an area-minimizing manifold $N$ and
\begin{enumerate}
\item if $\lambda_k\rightarrow 0$ is a sequence such that $\lambda_k^{-1}N$ converges weakly to a minimizing cone $\mathbf C$. The cone is called a \emph{Tangent cone at $0$}. 
    \item if $\lambda_k\rightarrow\infty$ is a sequence such that $\lambda_k^{-1}N$ converges weakly to a minimizing cone $\mathbf C$. The cone is called a \emph{Tangent cone at $\infty$}. 
\end{enumerate}
\end{definition}
The tangent cones of a manifold will help describe the local and global properties of the manifold. The aim of studying these is to hopefully understand local properties using the tangent cone at 0 and explore more global questions on $N$ using the tangent cone at $\infty$. 
\begin{theorem}[Almgren \cite{almgren1966}, Fleming \cite{Fleming1962OnTO}, Simons \cite{SimonJ}]
If $\mathbf C^n \subset \mathbb R^{n+1}$ is an area-minimizing cone and $2\leq n+1\leq 7$, then $\mathbf C$ is flat. 
\end{theorem}
The above theorem will imply that if $N^n\mathbb \subset R^{n+1}$ is an area-minimizing hypersurface, then $\dim( sing(N))\leq n-7$. In particular, any area minimizing hypersurface $N^n$ with $n\leq 6$ is smooth. The Simons cone is defined as
$$\mathcal{C}^{3, 3}\coloneqq C(S^3\times S^3)= \{ (x,y)\in \mathbb{R}^4\times \mathbb{R}^4: |x|^2=|y|^2\}\subset \mathbb R^8, $$
it is a cone with link $S^3\times S^3$ and $sing(\mathcal{C}^{3, 3})=\{0\}$. The Simons cone was the first example of an area-minimizing minimal surface with singularity pointed out by \cite{SimonJ} and Bombieri-De Giorgi-Giusti \cite{Bombieri1969}.
\begin{definition}[Quadratic cones]
\label{prel:defn:Quadratic cones}
Define quadratic cones, also known as generalized Simons cones, as 
$$\mathcal C^{p, q}\coloneqq C(S^p\times S^q)=\{(x, y)\in\mathbb R^{p+1}\times \mathbb R^{q+1}: q|x|^2=p|y|^2\}\subset \mathbb R^{n+1},$$ where $n-1=p+q$. 
\end{definition}
The cone $\mathcal C^{p, q}$ will be minimal for all $(p,q)$, and area-minimizing when $p+q>6$ or when $(p,q)=(3,3), (2,4), (4,2)$. For ease of notation through this paper, we will denote $\mathcal{C}\coloneqq \mathcal C^{p,q}$ for one of the area-minimizing quadratic cones. Since we are interested in the closed unit ball $B_1=\{x\in \mathbb R^{n+1}: \|x\|^2\leq 1\}$, denote 
$$\mathcal{C}_1\coloneqq \mathcal{C}\cap B_1 .$$

\subsection{Weighted spaces over $\mathcal{C}_1$} \label{WeSpC}

We shall be studying the minimal surface equation over $\mathcal C_1$ in some weighted H\"older spaces. The aim of constructing these weighted H\"older spaces on $\mathcal{C}_1$ is to study the graphs of functions on $\mathcal{C}_1$, and the surfaces corresponding to these function. For the graph to be well-defined near the origin, the function on $\mathcal{C}_1$ will have to decay faster than $|x|$ as it near the origin. With respect to the minimal surface equation, we will have some lower order decays and they in spirit will correspond to surfaces near $\mathcal{C}_1$. 
\begin{definition}[Weighted H\"older norm on $\mathcal C_1$]
Let $f:\mathcal{C}_1\rightarrow \mathbb{R}$, for $0<r\leq\frac{1}{2}$, and let $f_r:\mathcal{C}\cap (B_2\backslash B_1) \rightarrow \mathbb{R}$ be $f_r(x)\coloneqq f(rx)$. The \emph{weighted H\"older norm} on $\mathcal{C}_1$ is 
\begin{equation*}
\|f\|_{C^{k, \alpha}_{\delta}(\mathcal{C}_1)}\coloneqq \sup_{r\leq\frac{1}{2}}\|r^{-\delta}f_r\|_{C^{k, \alpha}(\mathcal{C}\cap (B_2\backslash B_1))}.
\end{equation*}
We say that $f\in C^{k, \alpha}_{\delta}(\mathcal{C}_1)$ if  $\|f\|_{C^{k, \alpha}_{\delta}(\mathcal{C}_1)}$ is finite.     
\end{definition}
The rescaled function $f_r$ maps the annulus $B_2\backslash B_1 \rightarrow B_{2r}\backslash B_r$, and $r^{-\delta}$ restricts the rate at which the function decays as we approach the origin, which will in turn  limit the class of functions we are working with. To demonstrate this, let $\|f\|_{C^{k, \alpha}_{\delta}(\mathcal{C}_1)}\leq C$ then $|f(x)|\leq C|x|^{\delta} $. Hence, the function grows at a rate bounded by $\delta$.
\begin{prop}
The weighted H\"older space $C^{k,\alpha}_\delta (\mathcal C_1)$ satisfies the following properties,
\begin{enumerate}
\item $C^{k, \alpha}_\delta(\mathcal{C}_1)$ is a Banach space.
\item If $f\in C^{k, \alpha}_{\delta}(\mathcal{C}_1)$, $i\leq k$, then for fixed $0<r<\frac{1}{2}$,
$$\nabla^i (r^{-\delta}f(rx))=r^{-\delta+i}(\nabla^i f)(rx),$$
hence $\nabla^i f \in C^{k-i, \alpha}_{\delta-i}(\mathcal{C}_1) $ and, 
$$\|\nabla^i f\|_{C^{k-i, \alpha}_{\delta-i}} \leq C \|f\|_{C^{k, \alpha}_{\delta}}.$$
\item If, $k+\alpha\leq l+\beta$, and $\delta \leq \delta'$ then 
$$ C^{l, \beta}_{\delta'}(\mathcal{C}_1)\subseteq C^{k, \alpha}_\delta(\mathcal{C}_1). $$
\end{enumerate}
\end{prop}
The second property implies that the linear operator $\Delta_{\mathcal{C}}: C^{2, \alpha}_{\delta}(\mathcal{C}_1)\rightarrow C^{0, \alpha}_{\delta-2}(\mathcal{C}_1)$ is well-defined, where $\Delta_{\mathcal{C}}$ is the Laplacian on $\mathcal{C}$. Since we are interested in graphs over $\mathcal{C}_1$, we will primarily use with $\delta>1$  and it will also be less than a constant which we shall define ahead.\\
Since, $\mathcal{C}$ is scale-invariant we have the following scaling properties for $\mathcal{L}_{\mathcal{C}}$, denoting $u_r(x)=u(rx)$, then
\begin{equation}\label{scaling1}
\mathcal{L_{C}}(u_r)=r^2(\mathcal{L_{C}}u)_r,
\end{equation}
and $\mathcal{M_{C}}$ satisfies a similar scaling property, 
\begin{equation}\label{scaling2}
(\mathcal{M_{C}}u)_r=\frac{1}{r}\mathcal{M_{C}}\left(\frac{1}{r}u_r\right).
\end{equation}

\subsection{Some minimal surfaces near quadratic cones}\label{CHS}
In this subsection we shall now study the minimal surface equation over $\mathcal C$, refer \cite{CHS} and Simon-Solomon \cite{SimonSolomon}. Using the second fundamental form of $\mathcal C$, eq.(\ref{L_N}) can be written as, 
$$\mathcal{L}_{\mathcal{C}}=\Delta_{\mathcal{C}}+\frac{(n-1)}{|x|^2}. $$
Let us denote the link of $\mathcal{C}$ as, 
$$ \Sigma\coloneqq {\mathcal{C}} \cap \partial B_1=\sqrt{\frac{p}{n-1}}S^p \times \sqrt{\frac{q}{n-1}}S^q. $$
Given $x\in \mathcal{C}$ we can write $x=r \omega$, where $r\in (0, \infty)$ and $\omega \in \Sigma.$ Rewriting $\mathcal{L}_{\mathcal{C}}$ in these polar coordinates we get, 
\begin{equation}\label{L_C polar}
\mathcal{L}_{\mathcal{C}}=\frac{\partial^2}{\partial r^2}+\frac{n-1}{r}\frac{\partial}{\partial r}+\frac{1}{r^2}\mathcal{L}_{\Sigma},
\end{equation}
where $\mathcal{L}_{\Sigma}= \Delta_{\Sigma} +(n-1).$ Since $\Sigma$ is compact, we can decompose $L^2(\Sigma)$ using an orthonormal basis $\{ \phi_i\}_{i=1}^{\infty}$ and corresponding eigenvalues $\mu_1\leq \mu_2 \leq \dots \rightarrow \infty,$ such that
\begin{equation}\label{eigenvalue}
\mathcal{L}_{\Sigma} \phi_i=-\mu_i \phi_i,
\end{equation}
and, $\langle\phi_i, \phi_j\rangle=\delta_{ij}$.
If we decompose $u\in C^{2, \alpha}_\delta(\mathcal{C}_1)$ in polar coordinates using the above basis, we get
$$u(r\omega)=\sum_{i=1}^{\infty}a_i(r) \phi_i(\omega). $$
In combination with the polar form eq.(\ref{L_C polar}), $\mathcal{L_C}u=0$ is equivalent to
\begin{equation}\label{homoDE}
r^2 a^{\prime\prime}_i(r)+(n-1)ra_i^{\prime}(r)-\mu_ia_i(r)=0.
\end{equation} 
Solving this homogeneous differential equation, we obtain a solution of the form $a_i(r)=r^{\gamma_i^+}$ or $r^{\gamma_i^-}$, where
\begin{equation}\label{gamma}
\gamma_i^{\pm}=\frac{-(n-2)\pm \sqrt{(n-2)^2+4\mu_i}}{2}.
\end{equation}
The key idea to solving this homogeneous differential equation is to use the spectrum of $S^p$ and $S^q$ to deduce the eigenvalues of the operator $\mathcal L_\Sigma$. Let $\nu_j^p$ denote the $j$th eigenvalue of the Laplacian on $\sqrt{\frac{p}{n-1}} S^p$ and it is given by 

$$\nu_j^p=\frac{(j-1)(n-1)(j+p-2)}{p}; \: j=1, 2, \dots.$$
The eigenspace corresponding to $\nu_j^p$ is $E_j^p\coloneqq\{$degree $j-1$ harmonic function on $\mathbb R^{p+1}$ restricted to the sphere\}. Let $\psi_j^p\in E^p_j$ and $\psi_k^q\in E^q_k$ then,
\begin{equation}
    \Delta_{\Sigma} (\psi^p_j\psi^q_k)=(\nu_j^p+\nu_k^q)\psi_j^p\psi_k^q.
\end{equation}
Hence, $\nu_j^p+\nu_k^q$ is an eigenvalue of $\Delta_{\Sigma} $ and its eigenvector is $\psi_j^p\psi_k^q.$ Thus the set of eigenvalues of  $\Delta_{\Sigma}$ is given by
$$\sigma(\Delta_\Sigma)=\{\nu_j^p+\nu_k^q: j,k=1,2,3, \dots \}.$$
We could get the same eigenvalue for different combinations of $j$ and $k$. So, the eigenspace will be the direct sum of all these combinations. The eigenspace corresponding to the $i^{th}$  eigenvalue of $\Delta_\Sigma$ is,
$$E_i(\Delta_{\Sigma})=\bigoplus_{I_i} E_j^p\otimes E_k^q$$
where $I_i=\{(j,k): \nu_j^p+\nu_k^q= i^{th} \mbox{eigenvalue of }\Delta_{\Sigma}  \} $. As above let $\mu_i$ represent the eigenvalue of $\mathcal{L}_{\Sigma}=\Delta_{\Sigma}+(n-1)$, then $\mu_i=\nu_j^p+\nu_k^q-(n-1)$. We can write down the first few eigenvalues and the corresponding $\gamma_i$'s from eq.(\ref{gamma}) as shown below,

\begin{table}[H]
    \centering
    \begin{tabular}{|| c ||c  |c |c||} \hline 
 i& (j,k)& $\mu_i$ &$\gamma_i^{\pm}$\\ \hline  \hline
 1& (1,1)& $-(n-1)$&$\frac{-(n-2)\pm \sqrt{(n-2)^2-4(n-1)}}{2}$\\ \hline   
 2& (1,2), (2,1)& 0 &0, $-(n-2)$\\ \hline 
 3& (2,2)&$n-1$&$1$, $-(n-1$)\\ \hline
 4& (3,1)&$ \left(1+\frac{2}{\max\{p,q\}}\right)(n-1)$  & $\gamma^+_4>1, \gamma^-_4<-2$ \\\hline
 \end{tabular}
    \caption{Eigenvalues of $\mathcal{L}_{\Sigma}$}
    \label{table}
\end{table}
If $i\geq 4$ then $\gamma^+_i>1$ and function will make graphical sense and hence it and we can solve for $u\in C^{k, \alpha}_\delta(\mathcal C_1)$, but what about the first three modes. The below result due to Simon-Solomon computes $\phi_i$ for $i=1,2,3$, which, combined with the values of $\gamma_i^+$, should tell us what surfaces they correspond to. 
\begin{theorem}[\cite{SimonSolomon}]\label{eigenvector} Let $\phi_i\in E_i(\mathcal{L}_\Sigma)$, i.e, $\mathcal L_\Sigma \phi_i=\mu_i\phi_i$,
\begin{enumerate}
\renewcommand{\labelenumi}{(\roman{enumi})}
    \item $\phi_1(w)=c$, for $c\in \mathbb R\backslash \{0\}$,
    \item  $\phi_2(w)=a.\nu(w)$, for $a\in \mathbb R^{n+1}\backslash\{ \mathbf 0\}$ \emph{(Translations)},
    \item $\phi_3(w)=(Aw).\nu(w)$, for $A\in \mathfrak{so}(n+1)$ \emph{(Rotations)}. 
\end{enumerate}
\end{theorem}
\begin{proof}
As above, let $E_i^p, E_i^q$ for $i=1,2$ denote the $i$th eigenspace of $\Delta_{\sqrt{\frac{p}{n-1}}S^p}$ and $\Delta_{\sqrt{\frac{q}{n-1}}S^q}$ respectively. The proof uses the knowledge we have about $E_1^p$ and $E_2^p$, the first two eigenspaces of the Laplacian on $\sqrt{\frac{p}{n-1}}S^p$ in combination with Table \ref{table}. For $w\in \Sigma$, let $w=(w_1, w_2)\in \sqrt{\frac{p}{n-1}}S^p\times \sqrt{\frac{q}{n-1}}S^q$ . Then, $\nu_{\mathcal{C}}(w)=(-w_2, w_1)$ is the normal vector at $w\in \mathcal C$. 
    \begin{enumerate}
        \item  $i=1$, i.e, $(j,k)=(1,1);$ corresponds to both constants functions on $S^p$ and $S^q$. Hence, $\phi_1(w)=c$, constant. 
        \item  $i=2$, i.e, $(j,k)=(1,2)$ or $(2,1);$ hence the eigenvector $\phi_2$ lies in  $(E_1^p \otimes E_2^q) \oplus (E_2^p\otimes E_1^q)$. For $a_1 \in E_1^p$and $a_2\in E_1^q$, then $\phi_2$ is of the form,  $$\phi_2(w)=a_1\cdot w_2+a_2\cdot w_1= a\cdot \nu(w),$$
where $a=(-a_1, a_2)\in \Sigma$. 
        \item  $i=3$, i.e, $(j,k)=(2,2);$ hence $\phi_3$ lies in $E_2^p\otimes E_2^q$, and hence, 
    $$\phi_3(w)=(a_1\cdot w_1)(a_2\cdot w_2)$$
we can rewrite this in the form of $(Aw)\cdot \nu(w)$, where $A$ can be constructed using $a_1$ and $a_2$ where $A \in \mathfrak{so}(n+1)$. 
    \end{enumerate}
\end{proof}
From Table \ref{table} and the above theorem note that, 
\begin{itemize}
    \item for $i=1$ where $-2\leq \gamma_1^+<0$, and $\phi_1=$ constant, with equality occuring for $\mathcal{C}^{3,3}$. This case corresponds to Hardt-Simon foliations (Section \ref{sec:HSFoliations}).
    \item  for $i=2$ where $\gamma_2^+=0$ and $\phi_2(w)=a. \nu(w), a\in \mathbb R^{n+1}\backslash \{0\}$ and corresponds to translations. 
    \item for $i=3$ where $\gamma_3^+=0$ and $\phi_3(w)=(Aw).\nu(w)$, for $A\in \mathfrak{so}(n+1)$ and corresponds to rotations. 
\end{itemize}
Since only graphs that decay at a rate greater than 1 are well-defined hypersurfaces over $\mathcal C_1$, we want to project $C^{2, \alpha}(\Sigma)$ to $i\geq 4$ and work over that set.

\begin{definition} \label{defn:Pi}
    For $g\in C^{2, \alpha}(\Sigma)$, define $\Pi: C^{2, \alpha}(\Sigma)\rightarrow C^{2, \alpha}(\Sigma)$    
    $$\Pi g\coloneqq \sum_{i=4}^{\infty}\langle g, \phi_i \rangle \phi_i,$$
    and define $ H: C^{2, \alpha}(\Sigma) \rightarrow C^{2, \alpha}_{\delta}(\mathcal{C}_1)$,  
$$\label{Hg}  Hg \coloneqq\sum_{i=4}^{\infty}\langle g, \phi_i \rangle \phi_i r^{\gamma_i^+}.$$
\end{definition}
$\Pi$g is the projection in $C^{2, \alpha}(\Sigma)$ to the complement of a finite dimensional space spanned by three cases in Theorem \ref{eigenvector}. For $u\in C^{2, \alpha}_{\delta}(\mathcal{C}_1)$, $\Pi u$ is going to denote $\Pi (u|_{\Sigma})$. $Hg$ is then Jacobi field with boundary $\Pi g$ i.e, solves the linear problem $\mathcal{L_C}(u)=0$ on $\mathcal{C}$ and $ \Pi (Hg)=\Pi g$.\\
Throughout this paper, we are going to assume $1<\delta<\gamma_4^+$, we can prove all results for $\delta>\gamma_4^+$ as well but will have to choose a $\Pi$ accordingly. Solving the linear problem and studying the Jacobi fields gives insights into solving the nonlinear minimal surface equations. With the goal to find more examples of minimal hypersurfaces with isolated singularities, \cite{CHS} proved the following theorems, 

\begin{theorem}[{\cite[Corollary 1.2]{CHS}}] \label{LinearC}
Consider the linearized operator, $\mathcal{L}_{\mathcal{C}}:C^{2, \alpha}_{\delta}(\mathcal{C}_1)\rightarrow C^{0, \alpha}_{\delta-2}(\mathcal{C}_1)$, and let $f\in C^{0, \alpha}_{\delta-2}(\mathcal{C}_1)$, $g\in C^{2, \alpha}(\Sigma)$, where $1<\delta<\gamma_4^+$. Then there exists a unique $u\in C^{2, \alpha}_{\delta}(\mathcal{C}_1)$ such that, 
    \begin{alignat*}{2}
       \mathcal{L}_{\mathcal{C}}u &=f &\quad &\text{on }\mathcal{C}_1,\\
                \Pi u&=\Pi g &      &\text{on }\Sigma,
\end{alignat*}
and $u$ satisfies the following Schauder-type estimates, 
$$\|u\|_{C^{2, \alpha}_{\delta}}\leq C(\|f\|_{C^{0, \alpha}_{\delta-2}}+\|\Pi g\|_{C^{2, \alpha}}).$$
\end{theorem}
The proof uses the Fourier decomposition of $u$ and $g$ as mentioned above, in which it uses the solution of the homogeneous differential equation to solve the non-homogeneous differential equation, 
$$r^2 a^{\prime\prime}_i(r)+(n-1)ra_i^{\prime}(r)-\mu_ia_i(r)=r^2 f_j ,$$
where $f_j(r)=\int_\Sigma f(rw)\phi_i(w)dw$. Once you have this ODE solved, with the right choice of constants we can ensure that the boundary conditions are satisfied. 

\begin{theorem}[{\cite[Theorem 2.1]{CHS}}]\label{CHSthm}
Consider the quasi-linear mean curvature operator $\mathcal{M}_{\mathcal{C}_1}:C^{2, \alpha}_{\delta}(\mathcal{C}_1)\rightarrow C^{0, \alpha}_{\delta-2}(\mathcal{C}_1)$ as described above, and let $g$ on $\Sigma$ with $\|g\|_{C^{2, \alpha}}$ is small, then there exists a unique solution $u\in C^{2, \alpha}_{\delta}(\mathcal{C}_1)$ satisfying, 
 \begin{alignat*}{2}
       \mathcal{M}_{\mathcal{C}}u &=0 &\quad &\text{on }\mathcal{C}_1,\\
                \Pi u&=\Pi g &      &\text{on }\Sigma,
\end{alignat*}
and satisfies the estimate
$$\|u\|_{C^{2, \alpha}_\delta}\leq C(p,q,\delta, \alpha)\|g\|_{C^{2,\alpha}}. $$
\end{theorem}
 The proof uses a contraction mapping on the operator -"$\mathcal{L_{C}}^{-1}$"$\circ Q_{C}$ in an appropriate subset of $C^{0, \alpha}_{\delta-2}(\mathcal C_1)$. We are going to use similar methods in Theorem \ref{1}.

\subsection{Hardt-Simon Foliations}\label{sec:HSFoliations} 
Given the cone $\mathcal{C}$ and link $\Sigma$, notice that $\mathbb{R}^{n+1} \backslash \mathcal{C}$ divides $\mathbb R^{n+1}$ into two connected regions $E_{+}$, $E_-$, where 
\begin{align*}
E_+ &=\{(x, y)\in \mathbb R^{p+1} \times \mathbb R ^{q+1}: q|x|^2>p|y|^2 \},\\
E_- &=\{(x, y)\in \mathbb R^{p+1} \times \mathbb R ^{q+1}: q|x|^2<p|y|^2 \}.
\end{align*}
Choose an orientation on $\mathcal{C}$ such that the normal $\nu_{\mathcal C}(x,y)=(-y, x) $ points in $E_+$. Hardt-Simon showed that,

\begin{theorem}[\cite{SimonHardt}]
For an area-minimizing quadratic cone $\mathcal C=\mathcal C^{p,q}$, there exists a connected oriented smooth minimal hypersurface $S_+\subset E_+$. This surface is unique up to scaling, that is if $S\subset E_+$ is a connected smooth minimal hypersurface, then $S=\lambda S_+$ for some $\lambda>0$. 
\end{theorem}
Hardt-Simon had proven this for regular area-minimizing cones. For more general area-minimizing cones, Wang {\cite{wang2022}} recently has proven the above theorem but without the uniqueness part.\\ 
Define,
$$ S_{\lambda}=\left\{\begin{array}{ccc}
      \lambda S_+ \: &\lambda>0, \\
      \mathcal{C} \: &\lambda=0, \\
      |\lambda| S_- \: &\lambda<0.
 \end{array} \right. $$

Then $S_\lambda\subset E_+$ for $\lambda >0$ are area-minimizing hypersurface, and infact are the only smooth minimizing hypersurfaces completely contained in $E_+.$ For $\lambda>0$, $S_{\lambda}$ forms a foliation of $E_+$, and this can be seen using the property that if $\zeta\in E_+$, the ray $\{s \zeta: s>0\}$ intersects $S_+$ in exactly one single point and transversely and that $(s \zeta)\cdot \nu_{S_+}>0 $, where $\mu_{S_+}$ is the normal vector on $S_+$.\\
This surface ``corresponds" to the Jacobi field with a growth rate of $\gamma_1^+$ from Table \ref{table}. Explicitly there is a $R_0>0$, such that $S_+\backslash B_{R_0}$ is a graph over $\mathcal{C}$ with the leading term $r^{\gamma_1^+}$. After a normalization, we can write 
$$ S_+\backslash B_{R_0}=\mbox{graph}_{\mathcal{C}}\{ r^{\gamma_1^+}+O(r^{\gamma_1^+ -\epsilon})\},$$
where we have normalized $S_+$ such that the coefficient of $r^{\gamma_1^+}$ is 1. By scaling, we can use this to see that, for $\lambda>0$
\begin{equation}\label{graphS_lambda}
S_{\lambda}\backslash B_{\lambda R_0}=\mbox{graph}_{\mathcal{C}}\{ \lambda^{(-\gamma_1^+ +1)} r^{\gamma_1^+}+O(r^{\gamma_1^+ -\epsilon})\}.
\end{equation}

In fact, one can prove better estimates, for example in the case of $\mathcal{C}^{3,3}$, (refer Davini \cite{Davini}, Sz{\'e}kelyhidi \cite{Gabor2020}) 
$$ S_+\backslash B_{R_0}=\mbox{graph}_{\mathcal{C}}\{ r^{-2}+br^{-3}+O(r^{-9})\}$$
The above results, along with Table \ref{table}, say that apart from Jacobi fields with a rate greater than 1 (for $i\geq 4$), we have the Hardt-Simon foliations, translations and rotations corresponding to $\gamma_1^+$, $\gamma_2^+=0$, and $\gamma_3^+=1$ respectively. For the quadratic cones, we have the following global result due to Simon-Solomon, 
\begin{theorem}[\cite{SimonSolomon}, \cite{Mazet_2017}]
Let $N$ be an area-minimising hypersurface in $\mathbb R^{n+1}$ and let $\mathcal{C}$ be a tangent cone at $\infty$ of $N$. Then up to translations, rotations, and rescalings $N=\mathcal{C}$, $S_+$, or $S_-$.   
\end{theorem}

\section{Approximate surface}\label{setup}
In general, we would like to study the minimal surface equation over $S_\lambda$, but the boundary presents an issue. On $S_\lambda$, it is not easy to construct or solve for Jacobi fields and additionally if we scale solutions on $S_+$, we would end up scaling the boundary as well. So rather than working with graphs over $S_{\lambda}$, we will be working with an approximate surface that is close to it. Similar to \cite{CHS}, we will first solve the graphical minimal surface equation on this approximate surface. In this section, we shall define and study the approximate surface and setup the required notation. \\
The idea to construct the approximate surface $\tilde{S_{\lambda}}$ is that it looks like the Simons cone far away from the origin and be smooth like the $S_{\lambda}$ as we get closer to the origin. To construct the glued object $\tilde{S}_\lambda$, we will need the following cutoff functions, 
\begin{definition}\label{approxsurface}
Let $r_\lambda\coloneqq\lambda^{\frac{n-1}{n}}$, define the cutoff function, $\gamma_1: \mathbb{R}^+\rightarrow [0,1]$ smooth such that
$$ \gamma_1(r)=\left\{\begin{array}{cc}
      0 \mbox{ if }& r\geq 2r_{\lambda}  \\
      1 \mbox{ if }& r\leq r_{\lambda} 
 \end{array} \right. , $$
 and define $\gamma_2:\mathbb R^+\rightarrow [0, 1] $ where $\gamma_2\coloneqq 1-\gamma_1$.
\end{definition}
Recall from eq.(\ref{graphS_lambda}) we know that outside of a compact set, $S_{\lambda}$ is the graph of $r^{\gamma_1^+}+O(r^{\gamma_1^+ -\epsilon})$ over $\mathcal{C}.$

\begin{definition}[Approximate surface] \label{defn:approxsurface}
Define $\tilde S_\lambda$ in the three regions as follows, 
\begin{itemize}
    \item $2r_\lambda\leq r \leq 1 :$ $\tilde S_\lambda=\mathcal C_1$
    \item $r_\lambda \leq r \leq 2 r_\lambda$: $\tilde S_\lambda $ is the graph of  $\gamma_1(r)(\lambda^{(-\gamma_1^+ +1)} r^{\gamma_1^+}+O(r^{\gamma_1^+ -\epsilon})) $ over the annular region  $\mathcal C_1\cap \{r_{\lambda}<r<2r_{\lambda}\}$. 
    \item $0<r\leq r_\lambda:$ $\tilde S_\lambda= S_\lambda$

\end{itemize}
Refer to Figure \ref{approx surface}.
\end{definition}
This surface is well-defined for $|\lambda|\ll 1$. Observe that $\tilde S_\lambda \subset B_1$ is a smooth connected hypersurface and $\partial \tilde S_\lambda =\Sigma$. In simpler terms $\tilde S_\lambda$ can be thought of as ``$\gamma_1 S_{\lambda}+ \gamma_2 \mathcal{C}_1$".

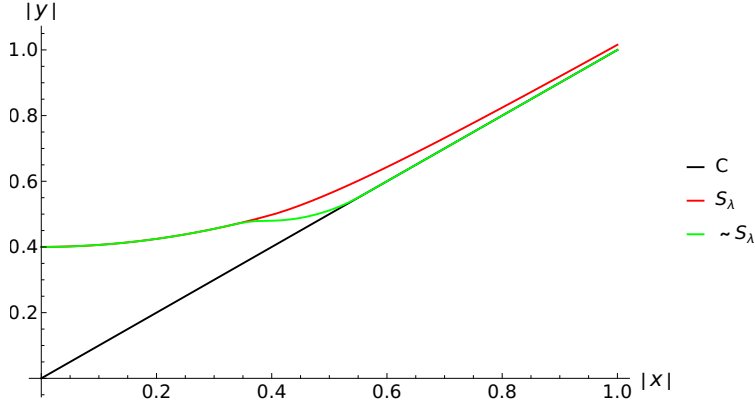
\begin{figure}
\centering
\begin{tikzpicture}
  \begin{axis}[
        axis lines = middle,
        xlabel = $|x| \in \mathbb R^{p+1}$,
        ylabel = {$|y|\in \mathbb R^{q+1} $},
       domain=0:4,
       ymin=0, ymax=4, 
      samples=200,
      xtick=\empty,
      ytick=\empty,
      width=12cm,
      height=8cm,
      legend pos= north east,
    ]
    \addplot[black, thick] {abs(x)};
    \addlegendentry{$\mathcal C^{p,q}$ }
    
    \addplot[blue, thick] {abs(x) + 1/(abs(x) + 1)};
    \addlegendentry{$S_+$}

    \addplot[green, thick] 
    {   (x < 0.8) * (x + 1/(abs(x) + 1)) +
        (x >= 0.8 && x < 2) * (x + ((2-x)/1.2)*(1/(abs(x) + 1))) +
        (x >= 2) * (x)
    };
    \addlegendentry{$\tilde S_\lambda$}

\end{axis}
\end{tikzpicture}
\caption{Approximate surface $\tilde S_\lambda$ (in green),\\
For $(x, y)\in \mathbb R^{p+1}\times\mathbb R^{q+1}$, the $x$-axis and $y$-axis represents $|x|$ and $|y|$.
respectively.}
\label{approx surface}
\end{figure} 

\subsection{Weighted spaces over $ \tilde{S_{\lambda}}$}
Similar to Section \ref{WeSpC}, in this subsection, we are going to define weighted H\"older norms on the approximate surface $\tilde{S_{\lambda}}$.

\begin{definition}[Weighted H\"older norm on $S_\lambda$]
    Let $u$ be a function on $S_{\lambda}$, and let $A_R=\{r: R<r<2R\}$ be an annulus, and $g_{euc}$ be the induced metric on $S_\lambda$, then the weighted H\"older norm on $S_\lambda$ is 
$$\|u\|_{C^{k, \alpha}_{\delta} (S_{\lambda})}= \sup_{R} R^{-\delta} \|u\|_{C^{k,\alpha} (S_{\lambda}\cap A_R, R^{-2}g_{euc})},$$
where the supremum is taken over $R$ such that the ball $B_{\frac{3R}{2}}(0)$ intersects $S_{\lambda}$.
\end{definition}
Using the above norm on $S_\lambda$ in combination with the weighted norm on $\mathcal C_1$, we can define a weighted H\"older norm on the approximate surface.
\begin{definition}[Weighted H\"older norm on $\tilde S_\lambda$]
Let $u$ be a function on $\tilde S_\lambda$, define 
$$\|u\|_{C^{k, \alpha}_{\delta} (\tilde{S}_{\lambda})}=\|\gamma_1 u\|_{C^{k, \alpha}_{\delta} (S_{\lambda})}+\|\gamma_2 u\|_{C^{k, \alpha}_{\delta} (\mathcal{C})}.$$
This is well-defined, let $u$ be a function on $\tilde S_\lambda$ then, $\gamma_1 u$ and $\gamma_2 u$ can be seen as functions on $S_{\lambda}$ and $\mathcal{C}$ respectively, and $C^{k, \alpha}_\delta(\tilde S_\lambda)\coloneqq \{u: \|u\|_{C^{k, \alpha}_{\delta} (\tilde S_\lambda)}< \infty\}$.
\end{definition}
Similar to weighted spaces over $\mathcal{C}$, on $\tilde{S}_{\lambda}$ we have, 
\begin{prop}
The weighted H\"older norm and space satisfy the following properties,
\begin{enumerate}
\item $C^{k, \alpha}_\delta(S_\lambda)$ and $C^{k, \alpha}_\delta(\tilde{S_\lambda})$ are Banach spaces.
\item If $f\in C^{k, \alpha}_{\delta}(\tilde{S_\lambda})$, then for fixed $0<r<\frac{1}{2}$,
$$\nabla^i (r^{-\delta}f(rx))=r^{-\delta+i}(\nabla^i f)(rx),$$
hence $\nabla^i f \in C^{k-i, \alpha}_{\delta-i}(\tilde{S_\lambda}) $ and,
$$\|\nabla^i f\|_{C^{k-i, \alpha}_{\delta-i}} \leq C \|f\|_{C^{k, \alpha}_{\delta}}.$$
\item Unlike $\mathcal C_1$, the weighted spaces on $\tilde{S}_{\lambda}$ are actually the same set of functions, only the norms differ and we have the following relations, for $\delta\leq \delta'$, 
$$ \|u\|_{C^{k, \alpha}_{\delta}(\tilde{S}_{\lambda})}\leq \|u\|_{C^{k, \alpha}_{\delta'}(\tilde{S}_{\lambda})}\leq C (\lambda)^{\delta-\delta'} \|u\|_{C^{k, \alpha}_{\delta}(\tilde{S}_{\lambda})},$$
where $C$ does not depend on $\lambda.$
\end{enumerate}
\end{prop}

\subsection{Fredholm Theory over Asymptotically Conical surfaces}
We are going to use results on asymptotically conical operators from Marshall \cite[Chapter 4]{Marshall} to study the operator, $\mathcal{L}_1:C^{2, \alpha}_{\delta}(S_+)\rightarrow C^{0, \alpha}_{\delta-2}(S_+)$. Which in turn will help us understand the operator $\tilde{\mathcal{L}_\lambda}$. Let $f$ be a function on $S_+$, define
$$\|f\|_{L^p_{k, \delta}}=\left(\sum_{j=0}^k \int_{S_+} r^{-n}|r^{j-\delta} \nabla^j f|^p dV\right)^{\frac{1}{p}},$$
and 
$$L^p_{k, \delta}(S_+)=\left\{ f: \left(\sum_{j=0}^k \int_{S_+} r^{-n}|r^{j-\delta} \nabla^j f|^p dV\right)<\infty \right\}.$$
For $p>1$, $L^p_{k, \delta}(S_+)$ is a Banach space and is reflexive. Moreover, $L^2_{k, \delta}(S_+)$ is a Hilbert space with the following inner product, 
$$\langle f, g\rangle\coloneqq \sum_{j=0}^k \int_{S_+} r^{2(j-\delta-n)}\langle\nabla^j f, \nabla^j g \rangle dV.$$
We have the following $L_2$ inner product identification, 
\begin{align}
L^p_{k, \delta}(S_+)\times L^{p^{\prime}}_{k, -\delta-n}(S_+) &\rightarrow \mathbb R, \\
(f, g) &\mapsto \int_{S_+} fg.
\end{align}
and induce the following dual identification
$$\Phi: L^p_{0, \delta}(S_+)\rightarrow L^{p^{\prime}}_{0, -\delta-n}(S_+)^*,$$
where 
$$\Phi(f)(g)\coloneqq\int_{S_+} fg.$$

\begin{theorem}
    (Embedding theorem)
    For $\delta< \delta^{\prime}$, $l-\frac{n}{p}>k+\alpha$ there are continuous embeddings, $$L^p_{l,\delta}(S_+)\subseteq C^{k, \alpha}_{\delta}(S_+)\subseteq L^q_{k,\delta^{\prime}}(S_+). $$
    
\end{theorem}
Recall the Jacobi operator $\mathcal L_1: L^p_{k, \delta}(S_+) \rightarrow L^p_{k-2, \delta-2}(S_+)$ on the Hardt-Simon surface and recall $\mathcal L_1=\Delta+|A_{S_+}|^2. $ 
\begin{theorem}\cite{Marshall}
    For $\mathcal{L}_1: L^p_{k, \delta}(S_+)\rightarrow L^p_{k-2, \delta-2}(S_+)$, there exists a subset $\mathcal{D}(\mathcal{L}_1)\subset \mathbb R$ independent of $p$, $k$ such that $\mathcal{L}_1$ is Fredholm iff $\delta\in \mathbb R \backslash \mathcal{D}(\mathcal{L}_1)$. Also, the set $\mathcal{D}(\mathcal{L}_1)$ is countable and discrete. 
\end{theorem}
In our case, $\mathcal D(\mathcal L_1)$ can be explicitly computed, but we only need that $\mathbb R_{>1}\subset \mathbb R\backslash \mathcal D(\mathcal L_1)$.
Also, the kernel of the $\mathcal L_1: L^{p}_{k, \delta}(S_+)\rightarrow L^{p}_{k-2, \delta-2}(S_+)$  is independent of $k $ and hence we shall denote it by $\ker \mathcal (L_1)^p_\delta$ to denote the dependence on $p$ and $\delta$. 
\begin{theorem}\cite{Marshall}\label{ImgL1}
    Let $\mathcal{L}_1: L^p_{2, \delta}(S_+)\rightarrow L^p_{0, \delta-2}(S_+)$ be the bounded linear map as discussed above, and $\delta>0$ be such that $\delta \in \mathbb R \backslash \mathcal{D}(\mathcal L_1)$. Then
$$\mbox{Img}(\mathcal{L}_1)=(\ker(\mathcal{L}_1^*))^{\perp}=\{f\in  L^p_{k, \delta}(S_+):\langle f, h\rangle_{L_2}=0, \forall h \in \ker(\mathcal{L}_1^*)\}, $$
where $\mathcal L_1^*$ is the $L_2$ adjoint to $\mathcal L_1$.
\end{theorem}
All these results can now be proven for any bundle over a manifold with a conical end or an asymptotically conical metric. \\
In our case, note that if $\mathcal{L}_1:  L^p_{2, \delta}(S_+)\rightarrow L^p_{0, \delta-2}(S_+)$ is a linear operator, then the adjoint to $\mathcal{L}_1$ is defined as, $\mathcal{L}_1^*:L^{p^{\prime}}_{0, -\delta-n+2}(S_+)\rightarrow L^{p^{\prime}}_{2, -\delta-n}(S_+)$ where  it satisfies, 
$$\langle u,\mathcal{L}_1^*v\rangle_{L^2}=\langle \mathcal{L}_1 u,v\rangle_{L^2}.$$
Using the self-adjointness of the Laplacian operator, note that the Jacobi operator $\mathcal L_1$ will be self-adjoint. Also, if $n\geq 8$ and $\delta >1$, then $-n-\delta+2<\gamma_1^+$. 
 By Hardt-Simon's result, the only Jacobi fields with growth less than $\gamma_1^+$ are trivial. Hence, the $\ker (\mathcal L^*_1)= \ker (\mathcal L_1)^{p'}_{-\delta-n+2}$ is trivial. The above theorem, combined with the fact that the kernel is trivial, gives the following result.
 \begin{theorem}\label{Surjective}
 Let $1<\delta<\gamma_4^+ $, on $S_+$ the bounded linear operator $\mathcal{L}_1:C^{2, \alpha}_{\delta}(S_+) \rightarrow C^{0, \alpha}_{\delta-2}(S_+)$ is surjective. 
    
\end{theorem}

Let $J\subset C^{2, \alpha}_\delta (S_+)$ be the set of Jacobi fields on $S_\lambda$, i.e, for $\phi \in J$ then $\mathcal L_1(\phi)=0$.

\begin{prop}\label{InverserOnS_+}
There is a injective continuous linear map $\mathcal F:C^{0, \alpha}_{\delta-2}(S_+) \rightarrow J^\perp \subset C^{2, \alpha}_\delta(S_+)$ such that $\mathcal L_1 \circ\mathcal F= Id_{C^{0, \alpha}_{\delta-2}(S_+)} $
\end{prop}
\begin{proof}
From Theorem \ref{Surjective}, we have that for a given $f\in C^{0, \alpha}_{\delta-2}(S_+)$ there exists $u\in C^{2, \alpha}_\delta (S_+)$ such that 
$$\mathcal L_1 (u)=f.$$
Since we can write $u=u_1+\phi$ where $\phi\in J $ and $u_1\in J^\perp$. Suppose we have two $u_1$'s then the difference will be a Jacobi field and be in $J^\perp$ , and hence will be zero. As a consequence $u_1$ will be unique. The choice is canonical and hence will be continuous.  

\end{proof}

\section{Minimal surfaces near Hardt-Simon surfaces}
In Section \ref{CHS}, we stated Cafferali-Hardt-Simon's result where they solve a boundary value problem for the minimal surface equation over $\mathcal C_1$, except for a finite-dimensional set corresponding to the first three Fourier modes. Due to the results of Hardt-Simon, and Simon-Solomon at a linear level, we exactly know what these correspond to Hardt-Simon surfaces ($S_{\lambda}$), translations, and rotations of $\mathcal{C}_1$ (see Theorem \ref{eigenvector} and Section \ref{sec:HSFoliations}). The aim is to prove that for $\|g\|_{c^{k, \alpha}}$ small enough, we can find a hypersurface close to the $\mathcal{C}_1$ that agrees with graph$_{\Sigma}g$ on the boundary and that is graphical modulo translations, rotations and dilations.\\   
In this section, we prove a version of Theorem \ref{CHS} on $S_\lambda$ for $\lambda$ small. 
\begin{theorem}\label{1}

For $0<\lambda \ll 1$ and let $S_\lambda$ denote a leaf of the foliation on $\mathcal C^{p,q}$, let $\Sigma_\lambda\coloneqq S_\lambda \cap \partial B_1$ and $g\in C^{2, \alpha}(\Sigma_\lambda)$, then for $\|g\|_{C^{2, \alpha}}$ small and independent of $\lambda$,
    For $\lambda \ll 1$ let us denote $\Sigma_\lambda \coloneqq S_{\lambda}\cap \partial B_1$ and let $g\in C^{k, \alpha}(\Sigma)$ , then for $\|g\|_{c^{k, \alpha}}$ small (independent of $\lambda$), there exists $u\in C^{2, \alpha}(S_\lambda)$ such that
\begin{alignat*}{2}
       \mathcal{M}_\lambda u&=0&\quad &\text{on }S_{\lambda}\cap B_1\\
                \Pi u&=\Pi g &      &\text{on }\Sigma_\lambda,
\end{alignat*} 
recall $\Pi$ is the projection away from the finite-dimensional space as discussed in Section \ref{setup}.
\end{theorem}
\textbf{Note}: The projection operator $\Pi $ was initially defined $\Sigma$ and not on $\Sigma_\lambda $. However, since $\Sigma_\lambda$ is a constant graphical perturbation of $\Sigma$ (ref Section \ref{sec:HSFoliations}).  For a small $\lambda$ this will be an identification between $\Sigma$ and $\Sigma_\lambda$. So we extend $\Pi$ to $\Sigma_\lambda$ through this identification. We shall abuse this identification throughout this section.  \\

Through this section we are going to assume $\lambda>0 $, and for $\lambda<0$ the result will follow similarly. For $\lambda=0$ it is just Theorem \ref{CHSthm}. The strategy to solve will involve working over the approximate surface $\tilde{S_{\lambda}}$ instead of $S_{\lambda}$. Similar to \cite{CHS} we will solve the linear problem first and then use the linear version to solve the quasi linear equation over $\tilde{S}_{\lambda}$, which is intuitively close to $S_{\lambda}$. Working with $\tilde{S_{\lambda}}$ rather than $S_{\lambda}$  is to ensure that we have control over the boundary independent of the scaling $\lambda$. We shall use this theorem to prove the existence result by which we construct minimal surfaces near quadratic cones (Theorem \ref{MainTheo}).

\subsection{The Linear Equation}\label{Linear}
Let the mean curvature operator over $\tilde{S_{\lambda}}$ be denoted as $\tilde{\mathcal{M}_{\lambda}}$. Let $\tilde{\mathcal{L}_{\lambda}}:C^{2, \alpha}_{\delta}(\tilde{S_{\lambda}})\rightarrow C^{0, \alpha}_{\delta-2}(\tilde{S_{\lambda}})$ be the linearization of $\tilde{\mathcal{M}}_\lambda$ at zero (similar to Section \ref{sec:prelim}). We are first going to prove the following linearized result,
\begin{theorem}\label{Lineareqn}
Let $f\in C^{0, \alpha}_{\delta-2}(\tilde{S_{\lambda}}), g\in C^{2, \alpha}(\Sigma)$ and $1<\delta<\gamma_4^+$, then there exists $u\in C^{2,\alpha}_{\delta}(\tilde S_\lambda)$ such that it solves,   
\begin{alignat*}{2}
       \tilde{\mathcal{L}_{\lambda}}u&=f &\quad &\text{on }\tilde{S_{\lambda}},\\
                \Pi u&=\Pi g &      &\text{on } \Sigma
\end{alignat*} 
and $u$ satisfies the following weighted Schauder boundary estimate
\begin{equation}\label{bdryschauder}
    \|u\|_{C^{2, \alpha}_{\delta}}\leq C(\|f\|_{C^{0, \alpha}_{\delta-2}}+\|\Pi g\|_{C^{2, \alpha}}).
\end{equation}
\end{theorem}
\begin{proof}
Define the following scaling map between leaves of the foliation,  \begin{eqnarray*}
 \Lambda:S_{+} &\rightarrow& S_{\lambda},\\
 x &\mapsto& \lambda x.
 \end{eqnarray*}
Recall the cutoff function $\gamma_1: \mathbb R\rightarrow [0,1]$,
 $$ \gamma_1=\left\{\begin{array}{cc}
      0 \mbox{ if } & r > 2r_{\lambda}  \\
      1 \mbox{ if } & r< r_{\lambda} 
 \end{array} \right.,$$
 and $\gamma_2 = 1-\gamma_1$. The preliminary guess would be to patch up solutions on the two regions $\mathcal{C}_1$ and $S_+$ together to get a solution to this problem. So if $f\in C^{0, \alpha}_{\delta-2}(\tilde S_\lambda)$, writing  
 $$ f = \gamma_1 f+\gamma_2 f ,$$
and this enables us to think of $\gamma_1 f$ and $\gamma_2 f$ as functions on $S_{\lambda}$ and $\mathcal{C}_1$. Given a function on $S_\lambda$, using $\Lambda$ we can rescale it to a function on $S_+$ and vice versa. Hence, we can write $f$ as,  
$$ f= \gamma_1 \lambda f_1\circ \Lambda^{-1}+ \gamma_2 f ,$$
here $f_1$ is a function on $S_+$. By Theorem \ref{InverserOnS_+}, choose  $u_1 =\mathcal F(f_1)\in C^{2, \alpha}_\delta(S_+)$ where $\mathcal{L}_1 u_1 =f_1$ on $S_+$. Using Theorem \ref{LinearC}, let $u_{2, g}\in C^{2, \alpha}_\delta (\mathcal C_1) $ be such that 
\begin{alignat*}{2}
       \mathcal{L}_{\mathcal{C}} u_{2, g}&=\gamma_2f &\quad &\text{on }\mathcal C_1,\\
                \Pi u_{2,g}&=\Pi g &      &\text{on } \Sigma.
\end{alignat*} 
The first idea would be to glue these solutions, i.e, $\gamma_1\lambda u_1\circ \Lambda^{-1}+\gamma_2 u_{2, g}$ to obtain a solution for the linearized problem on the approximate surface. But this does not solve the equation the gluing region $\{r_\lambda<r<2r_\lambda\}.$\\
The question can be reduced to finding a right inverse to $\tilde{\mathcal{L}_{\lambda}}$. Motivated by the methods in \cite{GaborBook}, \cite{Gabor2020},  we define a new set of cutoff functions on $\tilde{S_{\lambda}}$, which will help us construct an approximate inverse. Choose $a_-, a_0, a_+$ such that,
 $$0<a_- < a_0=\frac{n-1}{n}<a_+<1 ,$$
 and,
$$0<\lambda^{a_+}<r_{\lambda}=\lambda^{a_0}<2r_{\lambda}<\lambda^{a_-}<1.$$
Consider the smooth cutoff function $\beta: \mathbb{R}\rightarrow \mathbb{R} $ such that,
$$ \beta(x)=\left\{\begin{array}{cc}
      0 \mbox{ if }&  x< 0  \\
      1 \mbox{ if }& x> 1
 \end{array} \right. , $$
 and define 
 $$
 \psi_2(x)\coloneqq\frac{x-a_+}{a_0-a_+}.
$$
 Consider  $\beta_2:\tilde{S_{\lambda}}\rightarrow \mathbb{R}$ defined as, 
 $$ \beta_2(x)\coloneqq \beta\circ\psi_2\left(\frac{\log |x|}{\log \lambda}\right),$$
then,
$$ \beta_2(x)=\left\{\begin{array}{cc}
      0 \mbox{ if }& |x|< \lambda^{a_+}  \\
      1 \mbox{ if }& |x|> r_{\lambda}
 \end{array} \right. . $$
Similarly, define $\psi_1$ and $\beta_1$ such that
$$\beta_1(x)=\beta\circ\psi_1\left(\frac{\log |x|}{\log \lambda}\right),$$
and,
$$ \beta_1=\left\{\begin{array}{cc}
      1 \mbox{ if }& |x|< 2r_{\lambda}  \\
      0 \mbox{ if }& |x|> \lambda^{a_-} 
 \end{array} \right. . $$  
Note that $\beta_i$ is 1 on the support of $\gamma_i$ and hence $\beta_i\gamma_i=\gamma_i$. Caution for $i=1, 2$,  $\beta_i$'s do not form a partition of unity, but they are going to be important to define an approximate inverse by helping us bound the derivatives better in the gluing region. Using the derivatives of $\beta$ we can prove that $\beta_i$ satisfies the estimates
\begin{equation}\label{beta1}
     \|\nabla \beta_i\|_{C^{1, \alpha}_{-1}}\leq \frac{C}{\log \lambda}, 
\end{equation}
where the constant $C$ does not depend on $\lambda$. Now we use these $\beta_i$'s to define the following an approximate inverse. Define, $P_\lambda:C^{0, \alpha}_{\delta-2}(\tilde{S}_{\lambda})\times \Pi (C^{2, \alpha}(\Sigma)) \rightarrow C^{2, \alpha}_{\delta}(\tilde{S_{\lambda}})$ as
 $$P_\lambda(f, \Pi g)\coloneqq\beta_1 \lambda^2 u_1\circ \Lambda^{-1}+\beta_2 u_{2, g}.$$
Note, here $u_1, u_{2, g}$ are functions on $S_+$ and $\mathcal C_1$, then $\beta_1 u_1\circ \Lambda^{-1}$ and $\beta_2 u_{2, g}$ are functions on $\tilde{S_{\lambda}}$ and $\mathcal C_1$ respectively.  Similar to eq.(\ref{scaling1}), and (\ref{scaling2}), $\tilde{\mathcal{L}}_\lambda$ satisfies similar scaling properties, i.e, for $u_1$ a function on $S_+$ we have, 
\begin{equation}\label{scaling3}
\mathcal{L}_1u_1=\lambda^2\mathcal{L}_{\lambda}(u_1\circ\Lambda^{-1})\circ \Lambda.
\end{equation}
One would hope $P_\lambda$ is the right inverse to $\mathcal L_\lambda$, so let's try to compute $\tilde{\mathcal{L}_{\lambda}}\circ P_\lambda(f)-f$.  Let us begin by computing, $\mathcal{L}_{\lambda}(\beta_1 \lambda^{2}u_1\circ\Lambda^{-1})-\gamma_1 f$. Note,
\begin{align*}
&\mathcal{L}_{\lambda}(\beta_1 \lambda^2 u_1\circ\Lambda^{-1})-\gamma_1 f \\
&=\beta_1 \mathcal{L}_{\lambda} (\lambda^2 u_1\circ \Lambda^{-1})-\gamma_1 f +\Delta \beta_1 \lambda^2 u_1\circ \Lambda^{-1}+2\lambda^2\nabla \beta_1 \nabla( u_1\circ\Lambda^{-1}),  \\
&=\beta_1 \mathcal{L}_{1}(u_1)\circ \Lambda^{-1} -\gamma_1 f + \Delta \beta_1 \lambda^2 u_1\circ \Lambda^{-1}+2\lambda^2\nabla \beta_1 \nabla (u_1\circ\Lambda^{-1}),\\
&= \Delta \beta_1 \lambda^2 u_1\circ \Lambda^{-1}+2\lambda^2\nabla \beta_1 (\nabla u_1)\circ \Lambda^{-1},
\end{align*} 
where we have used the scaling from eq.(\ref{scaling3}), $\mathcal{L}_1(u_1)\circ\Lambda=\gamma_1f$, and $\beta_1$ is 1 on the support of $\gamma_1$. Using the bound on $\nabla \beta_1$ from eq.(\ref{beta1}), note that
\begin{align*}
&\|\mathcal{L}_{\lambda}(\beta_1 \lambda^2 u_1\circ\Lambda^{-1})-\gamma_1 f\|_{C^{0, \alpha}_{\delta-2}}\\
&\leq C\lambda^2\|\Delta \beta_1 \|_{C^{0, \alpha}_{-2}} \|u_1\circ \Lambda^{-1}\|_{C^{2, \alpha}_{\delta}}+ C\lambda^2 \|\nabla \beta_1 \|_{C^{1, \alpha}_{-1}}\| (\nabla u_1)\circ \Lambda^{-1}\|_{C^{1, \alpha}_{\delta-1}}, \\
&\leq C\lambda^{2+\delta}\|\Delta \beta_1 \|_{C^{0, \alpha}_{-2}} \|u_1\|_{C^{2, \alpha}_{\delta}}+ C\lambda^{1+\delta} \|\nabla \beta_1 \|_{C^{1, \alpha}_{-1}}\| \nabla u_1\|_{C^{1, \alpha}_{\delta-1}},\\
&\leq C \frac{\lambda^{2+\delta}}{\log \lambda}\|f\|_{C^{0, \alpha}_{\delta-2}} \leq o(1) \|f\|_{C^{0, \alpha}_{\delta-2}}.
\end{align*}
Using $\beta_2$, we can prove a similar inequality in the cone region as well,
\begin{equation*}
\mathcal{L_{C}}(\beta_2u_{2, g})-\gamma_2f=\beta_2\gamma_2 f- \gamma_2 f+2\nabla \beta_2 \nabla u_{2, g}+\Delta \beta_2 u_{2, g},
\end{equation*}
and,
\begin{equation}\label{eqn:LC bound}
\|\mathcal{L_{C}}(\beta_2u_{2, g})-\gamma_2f\|_{C^{0, \alpha}_{\delta-2}}\leq o(1)(\|f\|_{C^{0, \alpha}_{\delta-2}}+\|\Pi g\|_{C^{2, \alpha}}).
\end{equation}
Define,
\begin{align*}
    \mathbb{L_\lambda}: C^{2, \alpha}_{\delta}(\tilde{S_\lambda})& \rightarrow C^{0, \alpha}_{\delta-2}(\tilde{S_\lambda})\times \Pi (C^{2, \alpha}(\Sigma)),\\
    \mathbb L_\lambda(u)& \coloneqq(\tilde{\mathcal L}_\lambda(u), \Pi u).
\end{align*}
Hence, by combining the inequality on both the regions we get,
$$\|(\mathbb L_{\lambda}\circ P_\lambda -Id)(f, \Pi g)\|_{C^{0, \alpha}_{\delta-2}\times \Pi(C^{2, \alpha})}\leq o(1)(\|f\|_{C^{0, \alpha}_{\delta-2}}+\|\Pi g\|_{C^{2, \alpha}} ),$$
where $o(1)$ goes to zero as $\lambda\rightarrow 0$, and $Id$ is the identity function on $C^{0, \alpha}_{\delta-2}(\tilde{S}_{\lambda})\times \Pi (C^{2, \alpha}(\Sigma))$. By choosing $\lambda \ll 1$ we can ensure that, 
$$\|(\mathbb L_{\lambda}\circ P_\lambda -Id)\|_{C^{0, \alpha}_{\delta-2}\times\Pi(C^{2, \alpha})}\leq \frac{1}{2}.$$
This implies that for $\lambda$ small, $\mathbb L_\lambda\circ P_\lambda$ is invertible on $C^{0, \alpha}_{\delta-2}(\tilde{S}_{\lambda})\times \Pi (C^{2, \alpha}(\Sigma))$. Define,  
$R_{\lambda}:C^{0, \alpha}_{\delta-2}(\tilde{S}_{\lambda})\times \Pi (C^{2, \alpha}(\Sigma)) \rightarrow C^{2, \alpha}_{\delta}(\tilde{S_{\lambda}})$, as
\begin{equation}\label{R_lambda}
R_{\lambda}\coloneqq P_\lambda\circ (\mathbb L_{\lambda}\circ P_\lambda)^{-1}.
\end{equation}
\begin{prop}
The norm of  $R_\lambda$ is bounded and $R_\lambda$ is the right inverse to $\mathbb L_\lambda$, i.e, given $f\in C^{0, \alpha}_{\delta-2}(\tilde{S}_\lambda)$ and $g\in C^{2, \alpha}(\Sigma)$ then $u=R_\lambda(f,\Pi g)$ is a solution to Theorem \ref{Lineareqn}.  
\end{prop}
\begin{proof}
The norm of $(\mathbb L_\lambda\circ P_\lambda)^{-1}$ is bounded by 2, hence the norm of $R_{\lambda}$ is bounded independent of $\lambda$, and $\mathbb L_\lambda \circ R_\lambda =Id$. Therefore, if $u=R_\lambda (f, \Pi g)$ then $\mathcal{L}_\lambda u =f$ and $\Pi u= \Pi g$.  
\end{proof}
 The Schauder type estimates follow from $R_\lambda$ is bounded, if $u=R_\lambda (f, \Pi g)$, then
 $$ \|u\|_{C^{2, \alpha}_{\delta}}\leq C(\|f\|_{C^{0, \alpha}_{\delta-2}}+\|\Pi g\|_{C^{2, \alpha}})$$
\end{proof}
Since the choices $u_1$ and $u_{2, g}$ are unique and continuous, hence the inverse $R_\lambda (f, g)$ is continuous and satisfies
\begin{equation} \label{Inverse.Linearity}
 R_\lambda(f_1, g_1)+R_\lambda(f_2, g_2)=R_\lambda(f_1+f_2, g_1+g_2).
 \end{equation}

Jacobi fields, which are solutions to the linear problem, are our first approximate solutions to the non-linear problem. Define $H_\lambda g =R_\lambda (0, \Pi g)$ as the solution in Theorem \ref{Lineareqn} with $f=0$, i.e, 
\begin{alignat*}{2}
       \tilde{\mathcal{L}_{\lambda}}(H_\lambda g)&=0 &\quad &\text{on }\tilde{S_\lambda},\\
                \Pi (H_\lambda g)&=\Pi g &      &\text{on }\Sigma.
\end{alignat*} 
If $u=R_\lambda (f,\Pi g)$ is the solution to the linear problem as constructed in Theorem \ref{Lineareqn}, then by eq.(\ref{Inverse.Linearity}),  $R_\lambda (f, 0)=u-H_\lambda g $ . Hence, from the Schauder estimate we have the following, 
\begin{equation}\label{uHg bound}
\|u-H_\lambda g\|_{C^{2, \alpha}_{\delta}}<C\|f\|_{C^{0, \alpha}_{\delta-2}}.
\end{equation}

\subsection{The Nonlinear Equation}\label{Non Linear}
Through this subsection fix a $g\in C^{2, \alpha}(\Sigma)$. The aim is to solve the minimal surface equation $\tilde{\mathcal{M}}_\lambda u=0$  for a given boundary condition $\Pi g$ on $\Sigma$. Since, $\tilde{\mathcal{M}_{\lambda}}=\tilde{\mathcal{L}_{\lambda}}+\tilde{Q_{\lambda}}$, hence solving $\tilde{\mathcal{M}}_\lambda u=0$ is equivalent to proving 
\begin{align*}
    \tilde{\mathcal{L}_{\lambda}}u&=-\tilde{Q}_\lambda u,\\
    u&=-R_{\lambda}(\tilde{Q}_\lambda(u),\Pi g).
\end{align*}
With the above motivation, we define, 
\begin{align*}
\mathcal{N}:C^{2,\alpha}_{\delta}(\tilde{S_{\lambda}})&\rightarrow C^{2,\alpha}_{\delta}(\tilde{S_{\lambda}}),\\
u&\mapsto -R_{\lambda}(\tilde{Q_{\lambda}}(u), \Pi g).
\end{align*}
The aim is to prove that in a nice enough set, we can find a fixed point for $\mathcal{N}$ with the required boundary conditions. Consider the following lemma, 

\begin{lemma} \label{Q bound lemma} For $\|u\|_{C^{2, \alpha}_1}$ small
$$\|\tilde{Q}_\lambda(u)-\tilde{Q}_\lambda(v)\|_{C^{0, \alpha}_{\delta-2}}\leq C(\|u\|_{C^{2, \alpha}_{1}}+\|v\|_{C^{2, \alpha}_{1}})\|u-v\|_{C^{2, \alpha}_{\delta}}.$$
\end{lemma}

\begin{proof}
    By the mean value theorem, we have     $$\tilde{Q_{\lambda}}u-\tilde{Q_{\lambda}}v=D\tilde{Q}_{\lambda,\omega}(u-v),$$
    where $\omega=tu+(1-t)v$ for $0\leq t \leq 1.$ Here $D\tilde Q_{\lambda, w}$ is the derivative of the function $\tilde Q_\lambda$ at $w$, i.e, 
    $$D\tilde Q_{\lambda, w} (u)=\left.\frac{d}{dt}\right|_{t=0} \tilde Q_\lambda(w+tu). $$
Differentiating $\tilde{\mathcal{M}_{\lambda}}=\tilde{\mathcal{L}_{\lambda}}+\tilde{Q_{\lambda}}$, and noting that $D\tilde{\mathcal L}_{\lambda,w}=\tilde {\mathcal L}_\lambda = D\tilde{\mathcal M}_{\lambda,0}$ we get, 
\begin{align*}
        D\tilde{Q}_{\lambda,\omega}&=D\tilde{\mathcal M}_{\lambda,\omega}-D\tilde{\mathcal M}_{\lambda,0},\\
    \| D\tilde{Q}_{\lambda,\omega}\|_{C^{0, \alpha}}&=\|D\tilde{\mathcal M}_{\lambda,\omega}-D\tilde{\mathcal M}_{\lambda,0}\|_{C^{0, \alpha}}\leq C\|\omega\|_{C^{2, \alpha}}.
    \end{align*}
    Hence, 
    \begin{align*}
        \|\tilde{Q_{\lambda}}u-\tilde{Q_{\lambda}}v\|_{C^{0,\alpha}}&\leq C\|\omega\|_{C^{2, \alpha}}\|u-v\|_{C^{2, \alpha}},\\
        &\leq C(\|u\|_{C^{2, \alpha}}+\|v\|_{C^{2, \alpha}})\|u-v\|_{C^{2, \alpha}}.
    \end{align*}
    Now using the scaling, 
    $(\tilde{\mathcal{M}_{\lambda}}u)_r=r^{-1} \tilde{\mathcal{M}_{\lambda}}(r^{-1}u_r)$ and a similar scaling for $\tilde{\mathcal{L}_{\lambda}}$ we can derive that $\tilde{Q_{\lambda}}$ also scales similarly. Using the above inequality for $u_r$ and $v_r$ in place of $u$ and $v$, 
\begin{align*}
        r^{-\delta}\|(\tilde{Q_{\lambda}}u)_r-(\tilde{Q_{\lambda}}v)_r\|_{C^{0,\alpha}}&\leq C r^{-\delta}\| \tilde{Q_{\lambda}}u_r-\tilde{Q_{\lambda}}v_r\|_{C^{0, \alpha}} r^{-1},\\
        &\leq Cr^{-1}(\|u_r\|_{C^{2, \alpha}}+\|v_r\|_{C^{2, \alpha}})r^{-\delta}\|u_r-v_r\|_{C^{2, \alpha}}.
    \end{align*}
Using the definition for the norms we get, 
\begin{equation} \label{Q.1}
    \|\tilde{Q_{\lambda}}u-\tilde{Q_{\lambda}}v\|_{C^{0,\alpha}_{\delta-2}}\leq C(\|u\|_{C^{2, \alpha}_1}+\|v\|_{C^{2, \alpha}_1})\|u-v\|_{C^{2, \alpha}_{\delta}}.
\end{equation}
\end{proof}
The above lemma, in combination with the contraction principle on an appropriate set, allows us to find a fixed point for the operator $\mathcal N$. Fix $g\in C^{2, \alpha}(\Sigma)$ and let $\|g\|_{C^{2, \alpha}}\leq \epsilon$, consider the following set, 
$$E_g\coloneqq \{u\in C^{2, \alpha}_{\delta}(\tilde{S_{\lambda}}): \Pi u=\Pi g, \|u-H_\lambda g\|_{C^{2, \alpha}_{\delta}}\leq c\|g\|^2\}.$$
\begin{theorem}
    For $\epsilon$ small and  $c$ large and independent of $\epsilon$, $\mathcal{N}$ is a contraction on $E_g$, and $\mathcal{N}$ has a fixed point on $E_g$.
\end{theorem}
\begin{proof}
For $u\in E_g$ , and $\epsilon$ small we can ensure that $\|u\|_{C^{2, \alpha}_1}+\|H_\lambda g\|_{C^{2, \alpha}_1}$ is small enough such that using the Lemma \ref{Q bound lemma} and the fact that $R_{\lambda}$ is bounded we can ensure that,
 $$\|\mathcal{N}(u)-\mathcal{N}(H_\lambda g)\|_{C^{2, \alpha}_{\delta}}\leq \frac{1}{2}\|u-H_\lambda g\|_{C^{2, \alpha}_{\delta}}.$$
Applying Lemma \ref{Q bound lemma} for $H_\lambda g$ and $0$ we get that, 
$$\|\tilde{Q_{\lambda}}(H_\lambda g)\|_{C^{0,\alpha}_{\delta-2}}\leq C_1 \|H_\lambda g\|_{C^{2,\alpha}_{1}}\|H_\lambda g\|_{C^{2,\alpha}_{\delta}}.$$
By Schauder estimates we have $\|H_\lambda g\|_{C^{2,\alpha}_{\delta}}\leq C_2\|g\|_{C^{2, \alpha}}$. 
Computing using \ref{uHg bound}
$$ 
 \|\mathcal{N}(H_\lambda g)-H_\lambda g\|_{C^{2,\alpha}_{\delta}} \leq C_3 \| \tilde{Q_{\lambda}}\|_{C^{0, \alpha}_{\delta-2}}\leq C_1C_3(C_2)^2\|g\|^2 .$$
Choose $c$ such that, $C_1C_3(C_2)^2\leq c/2$, and 
\begin{align*}
    \|\mathcal{N}(u)-H_\lambda g\|_{C^{2,\alpha}_{\delta}} &\leq \|\mathcal{N}(u)-\mathcal{N}(H_\lambda g)\|_{C^{2,\alpha}_{\delta}}+\|\mathcal{N}(H_\lambda g)-H_\lambda g\|_{C^{2,\alpha}_{\delta}},\\
    &\leq \frac{1}{2}\|u-H_\lambda g\|_{C^{2,\alpha}_{\delta}}+ \frac{c}{2}\|g\|^2 ,\\
    &\leq c\|g\|^2.
    \end{align*}
Hence if $u\in E_g$ then $\mathcal N(u)\in E_g$. Finally, if $u,v \in E_g$ and if $\|u\|_{C^{2,\alpha}_1},  \|v\|_{C^{2,\alpha}_1}<C_4$ small, note $C_4$ depends on $\|g\|_{C^{2, \alpha}}$ and therefore can be controlled by $\|g\|_{C^{2, \alpha}}$. Now, 
$$\|\mathcal N(u)-\mathcal N(u)\|_{C^{2, \alpha}_\delta}=\|R_\lambda\circ (\tilde Q_\lambda(u)-\tilde Q_\lambda(v), \Pi g)\|_{C^{2, \alpha}_\delta}. $$
Using that $R_\lambda$ is bounded and Lemma \ref{Q bound lemma}, 
\begin{align*}
    \|\mathcal N(u)-\mathcal N(u)\|_{C^{2, \alpha}_{\delta}}&\leq C \|u-v\|_{C^{2, \alpha}_\delta}(\|u\|_{C^{2,\alpha}_1}+\|v\|_{C^{2,\alpha}_1}),\\
    &\leq \frac{1}{2}\|u-v\|_{C^{2, \alpha}_\delta},
\end{align*}
where in the second line we have chosen $\epsilon$ small and hence $C_4$ is small.
By the contraction principle, $\mathcal{N}$ has a fixed point $u$, and it satisfies 
\begin{alignat*}{2}
       \tilde{\mathcal{M}_{\lambda}}u&=0&\quad &\text{on }\tilde{S_\lambda},\\
                \Pi u&=\Pi g &      &\text{on }\Sigma.
\end{alignat*} 
Hence this proves Theorem \ref{1} but for the approximate surface.
\end{proof}
 Now we use the result for the approximate surface to describe the result on $S_{\lambda}$. To see this, first consider the region $r\leq r_{\lambda}$: here both $S_{\lambda}$ and $\tilde{S_{\lambda}}$ are the same and hence in this region any graph on $\tilde S_\lambda$ is also a graph over $S_\lambda$. For $r\geq 2r_{\lambda}$, $S_{\lambda}$ is a $\lambda^{(\gamma_1^+ +1)} r^{\gamma^+_1}$ graph over $\tilde{S_{\lambda}}$. For $\lambda\ll 1$, it is easy to see that a graph over $\tilde{S_{\lambda}}$ is also a graph over $S_{\lambda}$. In the gluing region $r_\lambda<r<2r_\lambda $, we can similarly write any graph over $\tilde{S_{\lambda}}$ as a graph over $S_{\lambda}$.\\
Let $w\in \Sigma$ and let $\tilde w \in \Sigma_\lambda$ be its corresponding point. If $g$ is a boundary value over $S_{\lambda}$, and let $\tilde{g}$ be the corresponding boundary value over $\tilde{S_{\lambda}}$. Then $g(w)-\tilde{g}(\tilde w)=a$, where $a$ is a constant. By the definition of $\Pi$, we will have that $\Pi g= \Pi \tilde{g}.$ Hence, the solution we obtained for the approximate surface is also a solution to the same problem over the Hardt-Simon Foliation. This proves Theorem \ref{1} for $\lambda\ll 1$. 

\subsubsection{Spherical graphs}
For some of the results ahead, it will be easier to work with spherical graphs. We shall describe Theorem \ref{1} for spherical graphs. In the case of quadratic cones, it is easy to define spherical graphs.\\
Let $w \in \Sigma$, and let $\nu_{\mathcal{C}}(w)$ be the normal vector at $w$ on $\mathcal{C}$. Then the spherical graph of $g$ at $w$ is given by $\exp_w (g(w) \nu_{\mathcal{C}}(w))$. The geodesic at $w$ in the normal direction is given by,  $$\gamma(t)\coloneqq \cos (t) w+ \sin (t) \nu_{\mathcal{C}}(w).$$
    So given $u\in C^{2, \alpha}_\delta(\mathcal C_1)$
    $$sp.graph_{\mathcal C_1}(u)=\{ \exp_x (u(x)\nu_{\mathcal C}(x))| x\in \mathcal C_1\}, $$
   and similarly for $g\in C^{2, \alpha}(\Sigma)$, the spherical graph is given by 
$$sp.graph_{\Sigma}(g)=\{ \exp_w (g(w)\nu_{\mathcal C}(w))| w\in \Sigma\}.$$
The idea is that since $\exp_w (g(w)\nu_{\Sigma}(w))=\cos (g(w))w+\sin (g(w))\nu_{\Sigma}(w)\approx w+g(w)\nu_{\Sigma}(w)$ for $|g|$ small. Using the exponential maps we can define spherical graphs over $S_\lambda$ as well. To define the spherical graph on $S_\lambda$, let $x\in S_\lambda$, the spherical graph is defined similarly as above but for the projection of $\nu_{S_\lambda}(w)$ onto $T_w(|x|S^n)$, i.e, for $u$ a function on $S_\lambda\cap B_1$
$$ sp.graph_{S_\lambda\cap B_1}= \{\exp_x(g(x)\mbox{proj}_{T_x (|x|S^n)}(\nu_{S_\lambda}(x))) \}.$$
Here, proj$_T \nu$ is the projection of $\nu$ onto the hypersurface $T$. \\
We can rewrite Theorem \ref{1} for spherical graphs as below,  

\begin{theorem}
For $\lambda \ll 1$, given $g\in C^{2, \alpha}(\Sigma_\lambda)$ small, then there exists $u\in C^{2, \alpha}_\delta(S_\lambda)\cap \overline B_1$ such that $N=sp.graph_{S_\lambda}(u)$ is minimal and $\Pi u=\Pi g.$
\end{theorem}
We shall use either of the results interchangeably.
 
\subsection{Existence}
Let $\mathcal{C}$ represent a generalized Simons cone and $\Sigma=\mathcal{C}\cap \partial B_1$ represent its link. Define, $\theta\coloneqq (a, Q, \lambda)\in \mathbb R^{n+1}\times SO(n+1)\times \mathbb R $, where $a, Q, \lambda$ corresponds to translations, rotations, and dilations of the Hardt-Simon surface. In the same spirit, define
$$\theta (\mathcal{C})\coloneqq a+Q (S_{\lambda}), $$ and $\|\theta\|\coloneqq |a|+ \|Q-Id\|+|\lambda |$. Define $V=\mathbb R^{n+1}\times SO(n+1)\times \mathbb R$ and if $\tilde V=\mathbb R^{n+1}\times \mathfrak{so}(n+1)\times \mathbb R$, for any $g\in \Pi^\perp (C^{2, \alpha} (\Sigma))$ then from Section \ref{sec:prelim} we have that $g=\sum_{i=1}^3 \phi_i(w)a_i(r)$ and using Theorem \ref{eigenvector} this induces an identification between $ \Pi^\perp (C^{2, \alpha}(\Sigma))$ and $\tilde V$. Since the exponential map on the Lie algebra $\exp: \mathfrak{so}(n+1)\rightarrow SO(n+1)$ is a local diffeomorphism, i.e, in a small neighborhood of $(0, Id, 0)\subset V$, $V$ and $\tilde V$ would be diffeomorphic. Hence for a small neighborhood of $\Pi^\perp (C^{2, \alpha} (\Sigma))$ we would have the following identification
$$\Pi^{\perp}(C^{2, \alpha}(\Sigma) \xhookrightarrow[\:]h V,$$
where $h^{-1}(\theta)$ is such that, 
$$sp.graph_{\Sigma}(h^{-1}(\theta))=\theta(\mathcal{C})\cap \partial B_1 .$$
For small enough $\|\theta \|$ and $\|g\|$, this identification shall allow us to interchangeably use $\theta$ or a function in $\Pi^\perp(C^{2, \alpha}(\Sigma))$ to represent the spherical boundary of the translated, rotated, and dilated surface. \\
Let $g\in C^{2, \alpha}(\Sigma)$ and $\theta \in V$. For $\|g\|$ and $\|\theta\|$ small we can define, $g_{\theta}\in C^{2, \alpha}(\theta(\mathcal{C})\cap \partial B_1)$
such that, 
$$sp.graph_{\Sigma}(g)=sp.graph_{\theta(\mathcal{C})\cap \partial B_1}(g_{\theta}).$$
In simpler terms $g_\theta$ is the boundary function for the surface corresponding to $g$ but over the translated, rotated and dialated surfaces. 
Finding a surface over $\mathcal{C}_1$ with boundary $g$ is equivalent to finding a surface over $\theta(\mathcal C)\cap B_1$ with boundary $g_\theta$. We solve the following Dirichlet boundary problem, it is related to the result of Edelen-Spolaor \cite{Nick}. 
\begin{theorem}[Graphical minimal surfaces near the Simons cone]\label{MainTheo}
    Let $\mathcal{C}_1=C^{p,q}\cap B_1$ be the Simons cone as defined above, given $g\in C^{2, \alpha}(\Sigma)$ where $|g|$ is small and $\Sigma=\mathcal{C}\cap \partial B_1$. Then there exists a $\theta\in \mathbb R^{n+1}\times SO(n+1)\times \mathbb R$ such that we have a solution $u\in C^{2, \alpha}_\delta(\theta(\mathcal C_1))$ to 
\begin{alignat*}{2}
       \mathcal{M}u&=0&\quad &\text{on }\theta(\mathcal{C})\cap B_1,\\
                u&=g_{\theta} &      &\text{on }\theta(\mathcal{C})\cap \partial B_1 .
\end{alignat*}
and hence we have a minimal surface $N\subset B_1$ such that $\overline{N}\cap \partial B_1 =sp.graph_{\Sigma}(g)$ and $N$ is graphical over the perturbation $\theta(\mathcal C).$ 
\end{theorem}
\begin{proof}
We shall begin by defining the following function,
\begin{eqnarray*}
    \Theta: C^{2, \alpha}\times V  &\rightarrow& V,\\
    (g, \theta) &\rightarrow& h(u_{\theta}|_\Sigma-g_{\theta}),
\end{eqnarray*}
where $u_{\theta}$ is the solution to Theorem \ref{1} with $g_{\theta}$ instead of $g$. Note that $\Pi u_{\theta}=\Pi g_{\theta}$ and hence $u_{\theta}|_\Sigma-g_{\theta} \in  \Pi^{\perp} (C^{2, \alpha}(\Sigma))$. The function measures the difference in the boundary in Theorem \ref{1}.  \\
Note that we have 
$$\Theta(0, 0)=0,\:  \Theta(0, \theta)=\theta.$$
To see the second property note if $g=0$ then $g_{\theta}\approx -h^{-1}(\theta)\in \Pi^{\perp}C^{2, \alpha}(\Sigma)$, hence $\Pi g_\theta =0$ and $u_\theta=0$. \\
For ease in notation, define $\Theta_g:V\rightarrow V$ as $\Theta_g(\theta)\coloneqq \Theta(g,\theta)$. We have that, 
$$\Theta_0=Id_V.$$
Since $V=\mathbb R^{n+1}\times SO(n+1)\times \mathbb R$ is a finite-dimensional vector space, say $V=\mathbb R^m$, then  
$$\Theta_0:S^{m-1}_{\epsilon'}\rightarrow V\backslash\{0\},$$
is a degree 1 map, where $S^{m-1}_{\epsilon'}$ is an $\epsilon'$ neighborhood of  the m-1 dim sphere $S^{m-1}.$ Since $\Theta_g$ is continuous in $g$, for small $g$ we have $\Theta_0$ is homotopic to $\Theta_g$. That is, for $|g|<\epsilon$ small $\Theta_{tg}$ will be a homotopy from $\Theta_0$ to $\Theta_g$. Hence, 
$$\Theta_g: S^{m-1}_{\epsilon'}\rightarrow V\backslash\{0\}$$
will be a well-defined degree 1 map. Hence for $\|g\|_{C^{2,\alpha}}$ small we have, 
$$0\in \Theta_g (B^m(0)). $$
Hence, there exists a $\theta$ for each $g$ such that $\Theta(g, \theta)=0.$
\end{proof}
\subsection{ Singular minimal surfaces near the Quadratic cones}
This subsection aims to demonstrate a possible approach to proving uniqueness near quadratic cones. One would expect a stronger result with uniqueness like Conjecture \ref{conjecture}. \\
Since the above argument is based on the degree of the map, it does not give us a unique $\theta$ corresponding to $g$. In general, we would like to use the implicit function theorem and prove that there is a unique continuous choice $\theta_g$ such that $\Theta(g, \theta_g)$=0. To use the implicit function theorem, we will need $C^1$ regularity of the $\Theta$ in $\theta$ variable, which we currently do not. That would entail a better understanding of the inverse $R_\lambda$.\\
In this subsection, we demonstrate a similar uniqueness result involving only the simpler case, translations, and rotations. It sketches proofs in the simpler case to illustrate how a stronger result could be established. \\
We reduce to the case only involving translations and rotations and have the follow result. Most of the proofs is a blueprint to how such a stronger regularity result would be proven. 
\begin{theorem}\label{thm 4.6}
Let $\tilde{\Pi}:C^{2, \alpha}(\Sigma)\rightarrow C^{2, \alpha}(\Sigma)$ be defined as, 
$$ g\mapsto \sum_{i=2}^{\infty}\langle g, \phi_i \rangle \phi_i. $$
Given $g\in C^{2, \alpha}(\Sigma)$ there exists a unique continuous choice of $\theta=(a, Q)\in \mathbb R^{n+1}\times SO(n+1)$ with respect to $g$ such that, 
\begin{alignat*}{2}
       \mathcal{M}u&=0&\quad &\text{on }a+Q(\mathcal{C}_1),\\
                \tilde{\Pi}u&=\tilde{\Pi}g_{\theta} &      &\text{on }(a+Q(\mathcal{C}_1))\cap \partial B_1.
\end{alignat*}
Equivalently, there is a n-dimensional minimal current N, such that on the boundary $N$ and the $sp.graph_{\Sigma}(g)$ differ by a constant (corresponding to Hardt-Simon Foliations).
\end{theorem}
\begin{proof}
Given $\theta=(a, Q)$, let $h_\theta:\Sigma\rightarrow \mathbb R$ be the boundary of $\theta(\mathcal C_1)$ i.e, 
$$sp.graph_\Sigma(h_\theta)=\theta(\mathcal C)\cap \partial B_1 .$$
Define,
$$\tilde{w}\coloneqq \exp_w (h_\theta(w) \nu_{\mathcal{C}}(w))= \cos (h_\theta(w))w+\sin (h_\theta(w))\nu_{\mathcal{C}}(w) \in \theta(\mathcal{C})\cap \partial B_1. $$
Let $\nu_{\theta(\mathcal{C})}(\tilde{w})$ denote the normal at $\tilde{w}$ on $\theta(\mathcal{C})$ where the orientation is preserved through the operation $\theta$.  

\begin{lemma}\label{gtheta}
    Let $g_{\theta}$ be defined as above, and if  $|\theta|< \epsilon, |g|< \eta$ small. Then,    $$g_{\theta}(\tilde{w})=g(w)-h_\theta(w)+E(g, \theta)$$
    where $|E(g, \theta)|\leq C\epsilon(\eta+\epsilon)$.
\end{lemma}
\begin{proof}
We shall first prove that,
$\|\gamma^{\prime}(h_\theta(w))-\nu_{\theta(\mathcal{C})}(\tilde{w})\|\leq C\epsilon$, where $C$ is a constant. Once we know this the lemma will follow. For ease of notation, if $w=(w_1, w_2)\in \mathbb R^{p+1}\times \mathbb R^{q+1}$ then define, 
$$Tw=(-w_2, w_1) .$$
We shall solve the lemma only for two easier cases where $\theta=(a, 0), (0, Q)$. Note, $\gamma^{\prime}(h_\theta(w))=-\sin h_\theta(w) w+\cos h_\theta(w) Tw$.
 If $\theta=(0, Q)$ and let $Q=\exp(A)$ for $A\in \mathfrak{so}(n+1) $then,   $$\nu_{\theta(\mathcal{C})}(\tilde{w})=QT(Q^\top\tilde{w}).$$
 If $\|\theta\|<\epsilon$,  then $\|Q-Id\|<\epsilon$. Using this we can compute that,  $|QT(Q^\top \tilde{w}-T\tilde{w})|<C\epsilon$. Using Theorem \ref{eigenvector}, we can show that $$\tan h_\theta(w)=Aw.Tw .$$ Now, $T\tilde{w}-\gamma^{\prime}(h_\theta(w))=2\sin h_\theta(w) w$  and $|\tan h_\theta(w)|<C\epsilon$. Hence, $\|T\tilde{w}-\gamma^{\prime}(h_\theta(w))\|<C\epsilon$.\\ In combination we obtain $\|\gamma^{\prime}(h_\theta(w))-\nu_{\theta(\mathcal{C})}(\tilde{w})\|\leq C\epsilon$.
The case when $\theta=(a, 0)$ has a very similar computation, and hence  $\|\gamma^{\prime}(h_\theta(w))-\nu_{\theta(\mathcal{C})}(\tilde{w})\|\leq C\epsilon.$\\
Using the above estimate, and since $sp.graph_{\Sigma}(g)=sp.graph_{\theta(\mathcal{C})\cap \partial B_1}(g_{\theta})$ the lemma follows. \\
For a general $\theta =(a, Q)$ one would have to compute $\nu_{\theta(\mathcal C)}$ similarly and estimate   $\|\gamma^{\prime}(h_\theta(w))-\nu_{\theta(\mathcal{C})}(\tilde{w})\|$. The process is similar to the above cases and works out similarly.  
\end{proof}
Similar to the general case define,
$$ \Theta(g, \theta)\coloneqq u_{\theta}|_\Sigma- g_{\theta},$$
where $u_\theta$ and $g_\theta$ are defined similarly to the general case. Note $g-g_\theta\approx \theta$, hence $\tilde \Pi g_\theta \approx \tilde\Pi g$ and using the continuity of the solution, we get $u_\theta \approx u $ and,
$$\Theta(g, \theta)\approx \Theta(g, 0)+\theta.$$
To see this consider, $g$ and $g_\theta$  as above, then 
$$\|Hg_{\theta}\|_{C^{2, \alpha}_{\delta}} \leq C \|g_\theta\|_{C^{2, \alpha}}^2<C\eta^2(\eta+\epsilon)^2.$$
Using this, we have
\begin{align}
\label{eqation thm 4.6}
\|\mathcal{M_C}(Hg_\theta)-\mathcal{M_C}(Hg)\|_{C^{0, \alpha}_{\delta-2}}&=\|Q_C(Hg_\theta)-Q_C(Hg)\|_{C^{0, \alpha}_{\delta-2}}\nonumber, \\
&\leq (\|g_\theta \|+\| g\|)|\theta|^2(\|g\|+\|\theta\|)^2 .
\end{align}
Hence,
\begin{equation}
    \Theta(g, \theta)=\Theta(g, 0) + \theta + O(\epsilon^2).
\end{equation}
In general one can prove a similar expansion centered at any small $\theta$. From this we can use the implicit function theorem to get uniqueness. 
\end{proof}
Note: The aim of Theorem \ref{thm 4.6} to show a rough sketch of one approach to showing such uniqueness, and hence lot of the arguments are very rough. \\
The key difficulty in extending such a result to all perturbations is extending eq.(\ref{eqation thm 4.6}). While extending the inequality to the singular direction $\lambda$, we were not able to get such a stronger regularity. 

\subsection{Minimal surfaces near the Simons cone}
We state the following conjecture, which adds a uniqueness to Theorem \ref{MainTheo}, (see also {\cite[P4.1]{wang2020deformationssingularminimalhypersurfaces}})
\begin{conjecture}\label{conjecture}
    Let $\mathcal C_1$ be a quadratic cone in $B_1$ and $\Sigma$ be its link  as defined above,let $g\in C^{2, \alpha}(\Sigma)$ where $|g|$ is small. Then, there is a continuous choice of $\theta_g=(a, Q, \lambda)\in \mathbb R^{n+1}\times SO(n+1)\times \mathbb R$ such that, there exists a minimal hypersurface $N$ spherically graphical over $\theta (\mathcal C_1)\coloneqq (a+Q.S_{\lambda})\cap B_1$and  satisfying $N\cap \partial B_1=sp.graph_{\Sigma}g$.
\end{conjecture}
One of the essential ideas towards solving the conjecture would be to have a better understanding of the inverse of the linearized operator $R_\lambda$ (eq.(\ref{R_lambda})). A stronger uniqueness result like this should help us understand the Plateau's problem better near such quadratic cones, and a better understanding of $\theta_g$ would help us know precisely when singularities are formed near these quadratic cones.

\printbibliography

\end{document}